\documentclass[12pt]{amsart}

\usepackage{fullpage, amsfonts, color,mathrsfs}
\usepackage[all]{xy}

\definecolor{hot}{RGB}{65,105,225}
\usepackage[pagebackref=true,colorlinks=true, linkcolor=hot ,  citecolor=hot, urlcolor=hot]{hyperref}

\def\bA{\mathbb{A}}
\def\bN{\mathbb{N}}
\def\bQ{\mathbb{Q}}
\def\bC{\mathbb{C}}
\def\bZ{\mathbb{Z}}
\def\bP{\mathbb{P}}
\def\bL{\mathbb{L}}
\def\lam{\lambda}
\def\al{\alpha}
\def\be{\begin{equation}}
\def\ee{\end{equation}}
\def\cM{\mathcal{M}}
\def\dd{\textnormal{d}}
\def\cO{\mathcal{O}}
\def\cD{\mathcal{D}}
\def\cR {\mathcal{R}}
\def\pa{\partial}
\def\Gr{{\rm{Gr}}}
\def\DR{{\rm{DR}}}

\theoremstyle{plain}
\newtheorem{thms}{Theorem}[section]

\newtheorem{cor}{Corollary}[subsection]
\newtheorem{prop}{Proposition}[subsection]
\newtheorem{lem}{Lemma}[subsection]
\newtheorem{ques}{Question}[section]

\theoremstyle{definition}
\newtheorem{rmk}{Remark}[subsection]
\newtheorem{rmks}{Remark}[section]

\title{Monodromy conjecture for semi-quasihomogeneous hypersurfaces}

\author{Guillem Blanco}\address{Department of Mathematics, KU Leuven, Celestijnenlaan 200B, 3001 Leuven, Belgium}\email{guillem.blanco@kuleuven.be}
\author{Nero Budur}\address{Department of Mathematics, KU Leuven, Celestijnenlaan 200B, 3001 Leuven, Belgium, and BCAM, Mazarredo 14, 48009 Bilbao, Spain}\email{nero.budur@kuleuven.be}
\author{Robin van der Veer}\address{Department of Mathematics, KU Leuven, Celestijnenlaan 200B, 3001 Leuven, Belgium}\email{robin.vanderveer@kuleuven.be}

\begin{document}

\maketitle

\begin{abstract}
We give a proof the monodromy conjecture relating the poles of motivic zeta functions with roots of $b$-functions for isolated quasihomogeneous hypersurfaces, and more generally for semi-quasihomogeneous hypersurfaces. We also give a strange generalization allowing a twist by certain differential forms.
\end{abstract}

\section{Introduction}

The strong monodromy conjecture of Igusa and Denef-Loeser predicts that the order of a pole of the motivic zeta function $Z^{mot}_f(s)$ of a nonconstant polynomial $f\in\bC[x_0,\ldots, x_n]$ is less than or equal to its multiplicity as a root of the $b$-function $b_f(s)$ of $f$. 
The conjecture is  open even in the case $f$ has isolated singularities. In this note we prove it for a subclass of isolated hypersurface singularities:

\begin{thms}\label{thmA}
The strong monodromy conjecture holds for semi-quasihomogeneous hypersurface singularities.
\end{thms}

Recall that a germ of holomorphic function on a complex manifold is said to define a {\it quasi-homogeneous hypersurface singularity} if it is analytically isomorphic to the germ at the origin of a weighted homogeneous polynomial. A hypersurface singularity is {\it semi-quasihomogeneous} if it is analytically isomorphic to the germ at the origin of a polynomial $f=f_d+f_{>d}$ where $f_d$ is a weighted homogeneous polynomial of degree $d$ with an isolated singularity at the origin, and $f_{>d}$ is a finite linear combination of   monomials of weighted degree $>d$. We call such polynomials $f$ {\it semi-weighted homogeneous of initial degree $d$}.

Theorem \ref{thmA} follows from the next one,   which we prove using  the main result of \cite{W} allowing  computations of motivic zeta functions from partial embedded resolutions:

\begin{thms}\label{thm1} Let $w_0,\ldots,w_n\in\bZ_{>0}^{n+1}$ be a weight vector. Let $f\in\bC[x_0,\ldots,x_n]$ be a semi-weighted homogeneous polynomial of initial degree $d$ with respect to these weights. Assume $f_d$ is irreducible (this is automatic if $n>1$). Then the poles of $Z^{mot}_f(s)$ are of order at most one and are contained in $\{-1, -\frac{w_0+\ldots+w_n}{d}\}$ if ${w_0+\ldots+w_n}\neq d$, otherwise $-1$ is the only pole of $Z^{mot}_f(s)$ and it has order at most two.
\end{thms}

A similar result holds if $\bC$ is replaced by a field of characteristic zero, see Remark \ref{rmkK}. The case $n=1$ and $f_d$ is reducible is also easy to deal with but it depends on some classification, see Remark \ref{rmkIrr}.

Both results, even for the isolated weighted homogeneous case, do not seem recorded in the literature. 

The version of Theorem \ref{thm1} for Igusa's local zeta functions is
 \cite[Theorem 3.5]{ZG}.  It is known that motivic zeta functions specialize to Igusa's $p$-adic local zeta functions. Thus Theorem \ref{thm1} implies the characteristic-zero case of \cite[Theorem 3.5]{ZG}, giving it a different proof. The version of Theorem \ref{thmA} for Igusa local zeta functions of semi-weighted homogeneous polynomials with an additional non-degeneracy assumption is \cite[Corollary 3.6]{ZG}.

\begin{rmks}
A homogeneous polynomial with an isolated singularity does not have to be Newton nondegenerate, e.g. $(x+y)^2+xz+z^2$, from \cite[1.21]{Ko}. Thus the existing results on  motivic zeta functions for nondegenerate polynomials do not suffice to prove any of the two theorems from above.
\end{rmks}

We also give a strange generalization of Theorem \ref{thmA}. Let $g(x_0,\ldots, x_n)$ be another polynomial. One has the twisted motivic zeta function $Z^{mot}_{f,g}(s)$ obtained by replacing the algebraic top-form $\dd x$ with $g\dd x$, see (\ref{eq2}).  One also has the  twisted $b$-function $b_{f,g}(s)$ obtained by replacing $f^s$ with $gf^s$, see (\ref{eqb}).  We denote by $(\pa f)$ be the Jacobian ideal of $f$ in the ring $\cO=\bC\{x_0,\ldots,x_n\}$ of convergent power series, that is, the ideal generated by the first order partial derivatives of $f$.


\begin{thms}\label{prop2}  With the assumptions of Theorem \ref{thm1}, let $g=x^\beta$ with $\beta\in\bN^{n+1}$ be a monomial. Let $l(\beta)=\sum_{i=0}^nw_i(\beta_i+1)/d$. Assume that 
\be\label{eqmon}
f_d\text{ contains no monomial }x_ix_j^{k}\text{ of }\text{weighted degree }d\text{ with }i\neq j\text{ and }k>0\text{ if }\beta_i\neq 0
\ee
and
\be\label{eqpa}
g+(\pa f)\not\subset \sum_{l(\gamma)> l(\beta)}\cO x^\gamma+(\pa f)
\ee
(if $f=f_d$ is weighted homogeneous (\ref{eqpa}) is equivalent to  $g\not\in (\pa f)$).
Then the order of a pole of $Z^{mot}_{f,g}(s)$ is less than or equal to its multiplicity as a root of $b_{f,g}(s)$.
\end{thms}

This provides the first cases for which a question of Musta\c{t}\u{a} \cite{M} has a positive answer.

Theorem \ref{prop2}  follows from properties of the microlocal $V$-filtration together with the following (not strange) generalization of Theorem \ref{thm1}:

\begin{thms}\label{prop1}
With the assumptions of Theorem \ref{thm1}, let $g=x^\beta$ with $\beta\in\bN^{n+1}$ be a monomial satisfying (\ref{eqmon}). Then the poles of $Z^{mot}_{f,g}(s)$ are of order at most one and are contained in $\{-1, -l(\beta)\}$ if $l(\beta)\neq 1$, otherwise $-1$ is the only pole of $Z^{mot}_{f,g}(s)$ and has order at most two.
\end{thms}

\begin{rmks}\label{rmk2} We explain why Theorem \ref{prop2} is a strange generalization of the strong monodromy conjecture. A twisted generalization of Theorem \ref{thmA} relating the poles of $Z^{mot}_{f,g}(s)$ to the roots of $b_{f,g}(s)$ is not true. For example, Theorem \ref{prop2} is not true for arbitrary monomials:

\begin{enumerate}
\item[(i)] Take $f=y^2-x^3$ with  weight vector $w=(2,3)$ for $(x,y)$, and let $g=y$. Then $-l(\beta)=-\frac{8}{6}$ is a pole of $Z^{mot}_{f,g}(s)$, but it is not a root of $b_{f,g}(s)=(s+1)(s+\frac{11}{6})(s+\frac{13}{6})$. Here $g\in (\pa f)$ fails (\ref{eqpa}). 
\item[(ii)]  Take $f=y^3-x^7+x^5y$ with weight vector $w=(3,7)$, and let $g=x^6$. Then $-l(\beta)=-\frac{28}{21}$ is a pole of $Z^{mot}_{f,g}(s)$, but it is not a root of $b_{f,g}(s)$, since one can check that $-\frac{29}{21}$ is the biggest root of $\frac{b_{f,g}(s)}{s+1}$. Here $g$ fails (\ref{eqpa}) in a more subtle way since $x^4y\in g+(\pa f)$ and $l(x^4y)=\frac{29}{21}>\frac{28}{21}$. 
\end{enumerate}
\end{rmks}

\begin{rmks}
 A (not strange) generalization of the strong monodromy conjecture was posed in  \cite{B-ls}: for any polynomials $f$ and $g$, the poles of $Z_{f,g}(s)$  should be roots of the monic polynomial $b(s)$ generating the specialization of the  Bernstein-Sato ideal $B_{(f,g)}\subset \bC[s_1,s_2]$ to $(s_1,s_2)=(s,1)$. (In the example (i) from Remark \ref{rmk2}, $b(s)= \prod_{k=6,8,10,11,13,14,16}(6s+k)$ so $-\frac{8}{6}$ is a root.) In fact, it is more generally conjectured in \cite{B-ls}  that the polar locus of the multi-variable motivic zeta function $Z_F^{mot}(s_1,\ldots,s_r)$ is contained in the zero locus in $\bC^r$ of the Bernstein-Sato ideal $B_F\subset\bC[s_1,\ldots,s_r]$ for any tuple of polynomials $F=(f_1,\ldots,f_r)$. 
\end{rmks}

\begin{rmks} Condition (\ref{eqpa}) on $g=x^\beta$ implies that $l(\beta)$ is a spectral number of $f$, see \ref{subHf} and \cite[1.6]{S+}. More generally, we say that $g\in\bC[x_0,\ldots,x_{n+1}]$ {\it achieves the spectral number} $\al>0$ of a polynomial $f$ with an isolated singularity if the class of $g\dd x$ is nonzero in $Gr_V^\al\Omega_f^{n+1}$, see \ref{subHf}. Then one can view the above results as partial confirmation of: \end{rmks}

\begin{ques}\label{que}
Let $f\in\bC[x_0,\ldots,x_{n+1}]$ be a semi-weighted homogeneous polynomial. Let $\al$ be a spectral number of $f$ at the origin. Does there exist $g\in\bC[x_0,\ldots,x_{n+1}]$ achieving the spectral number $\al$  such that the only non-integral pole of $Z_{f,g}^{mot}(s)$ is $-\al$? 
\end{ques}

\begin{rmks}\label{rmkQue}$\;$

(i) The eigenvalue version of the question is true for all polynomials $f$ with an isolated singularity: if $\lambda$ is an eigenvalue of the monodromy of $f$ at the origin, there exists $g\in\bC[x_0,\ldots,x_{n+1}]$ such that $Z_{f,g}^{mot}(s)$ with only one non-integral pole $-\al$ such that $e^{2\pi i\al}=\lambda$ by \cite{CV}. 

(ii) The $b$-function version of the question is not true for all polynomials $f$ with an isolated singularity: $-\frac{6}{13}$ is a root of $b_f(s)$ if $f=xy^5+x^3y^2+x^4y$, but $-\frac{6}{13}$ not a pole of $Z^{mot}_{f,g}(s)$ for any $g$ by \cite[Remark 3.1]{BV}. Here $f$ is not semi-weighted homogeneous and $\frac{6}{13}$ is not a spectral number of $f$.

(iii) The spectral version, namely Question \ref{que}, is not true for all polynomials $f$ with an isolated singularity: take $f = (y^2-x^3)^2-x^5y$, then $\frac{5}{12}$ is a spectral number at the origin. Here $\frac{5}{12}$ is also the log canonical threshold $lct(f)$. It is known, with the same proof as for $g=1$, that the maximal pole of $Z_{f,g}^{mot}(s)$ is the negative of
$$
lct_g(f)= \min\{\al > 0 \mid  g \not\in \mathcal J (f^\al)\}
$$
where $\mathcal{J} (f^\al)$ are the multiplier ideals of $f$, cf. \cite{M}, \cite{DM}. Thus the only $g$ with $Z_{f,g}^{mot}(s)$ having $-\frac{5}{12}$ as a pole must satisfy that $lct_g(f)=lct(f)$. Therefore $g(0)\neq 0$ and hence $Z_{f,g}^{mot}(s)$ has the same poles as $Z_f^{mot}(s)$. Since $f$ is an  irreducible plane curve with 2 Puiseux pairs, one can compute that  $-\frac{5}{12}$ and $-\frac{11}{26}$ are the only non-integral poles of $Z_f^{mot}(s)$.
\end{rmks}

\begin{rmks}
It is known that $\frac{b_f(s)}{s+1}$ is the minimal polynomial of the action of $s$ on $\tilde H''_f/s \tilde H''_f$,  where $\tilde H''_f$ is the saturation of the Brieskorn lattice, by a result of Malgrange and Pham, see \cite{S-mb}. In light of the strong monodromy conjecture, a natural question is if the canonical splitting of the class of $[\dd x]$ in $\tilde H''_f/s \tilde H''_f$ is a linear combination of classes $\omega_\al\in \tilde H''_f/s \tilde H''_f$, such that $\al$ are the non-trivial poles of $Z_{f}^{mot}(s)$ and $\omega_\al$
 are eigenvectors for $s$ with eigenvalue $\al$. While this is true in the isolated weighted homogeneous case, it is not true in general: example (iii) from Remark \ref{rmkQue} is a counterexample.
\end{rmks}

\noindent
{\bf Acknowledgement.}   Computations were performed with the help of Singular. J. Sebag informed us he has obtained a preciser version of Theorem \ref{thm1}, to appear in \cite{Sb}. We thank  M. Musta\c{t}\u{a}, J. Sebag, W. Veys for useful comments, and Universitat de les Illes Balears for hospitality during writing part of this article. G.B. was supported by an FWO postdoctoral fellowship. N.B. was supported by the grants Methusalem METH/15/026 from KU Leuven and G097819N from FWO. R.v.d.V. was supported by an FWO PhD fellowship.

\section{Motivic zeta functions}\label{secMot}

\subsection{Motivic zeta functions}

Consider two regular functions $f, g:X\to \bC$ on a smooth complex algebraic variety $X$, with $f$ non-invertible. Let  $\mu:Y\to X$ be  an embedded  resolution of $fg$. Let $E_i$ with $i\in J$ be the irreducible components of the pullback  $\mu^*(div(f))$ of the divisor of $f$, $E_I^\circ=\cap_{i\in I}E_i\setminus \cup_{i\in J\setminus I}E_i$ for $I\subset J$, and $\mu^*(div(f))=\sum_{i\in J}N_iE_i$. Let  $K_{\mu}-\mu^*(div (g))=\sum_{i\in J}(\nu_i-1)E_i$ where $K_{\mu}$ is the relative canonical divisor.

Define
\be\label{eq2}
Z^{mot}_{f,g}(s):= \bL^{-(n+1)}\sum_{\emptyset\neq I\subset J}[E_I^\circ]\prod_{i\in I}\frac{(\bL-1)\bL^{-(N_is+\nu_i)}}{1-\bL^{-(N_is+\nu_i)}}\ee
where $[E_I^\circ]$ is the class of $E_I^\circ$ in the localization $\cM=K_0(Var_\bC)[\bL^{-1}]$ of the Grothendieck ring of complex varieties along the class  $\bL=[\bA^1]$. The definition is independent of the choice of $\mu$. 

The smallest set $\Omega$ of rational numbers $-\frac{\nu}{N}$ with multiplicities such that $Z^{mot}_{f,g}(s)$ is a rational function in ${1-\bL^{-(Ns+\nu)}}$ with $\frac{\nu}{N}\in \Omega$ (with pole orders given by the multiplicities) over the ring $\cM[\bL^{-s}]$, is called the {\it set of poles} of $Z^{mot}_{f,g}(s)$. 

When $g=1$, $Z_{f,g}^{mot}(s)$ is denoted $Z_f^{mot}(s)$, the usual Denef-Loeser zeta function of $f$.

\subsection{Proof of Theorem \ref{thm1}}\label{subpfThm1} 

Let $\mu:Y\to \bC^{n+1}$ be the $w$-weighted blowup of the origin. Let $E$ be the exceptional divisor and $H$ the strict transform of $\{f=0\}$.  We show first that $\mu$ is an embedded $\bQ$-resolution of $f$, see \cite[\S 1.4]{W}. By definition this means that $E\cup H$ has $\bQ$-normal crossings, that is, it is locally analytically isomorphic to the quotient of a union of coordinate hyperplanes among those given by a local system of coordinates $t_0,\ldots, t_n$ with a
 diagonal action of a finite abelian subgroup of $G$ of $GL_{n+1}(\bC)$, i.e. locally
$$f\circ\mu = t_0^{N_0}\ldots t_n^{N_n}:\bC^{n+1}/G\to \bC$$
for some $N_i\in\bN$.

The exceptional divisor $E$ is isomorphic to
the $w$-weighted projective $n$-space
 $$\bP_w^n=(\bC^{n+1}\setminus 0)/\sim$$
with $(u_0,\ldots,u_n)\sim (\lam^{w_0}u_0,\ldots, \lam^{w_n}u_n)$ for all $\lam\in\bC^*$.
  Denote by $[u_0:\ldots:u_n]_w$ the class of a point in $\bP_w^n$. The chart $U_0=\{u_0\neq 0\}$ of $\bP^n_w$ is identified with the quotient
$$
\frac{1}{w_0}(w_1,\ldots, w_n)=\bC^n/\mu_{w_0}
$$
of the action of the $w_0$-roots of unity $\lam$ defined by
$$(x_1,\ldots, x_n)\mapsto (\lam^{w_1}x_1,\ldots, \lam^{w_n}x_n).$$
A similar description holds for the other charts.

Write $f=f_d+f_{>d}$, where $f_d$ is the degree $d$ term, so that the origin is an isolated singularity of $f$ and $f_d$.  Since we assumed $f_d$ is irreducible, $f$ is also since their singularity is isolated.  Thus the strict transform $H$ is irreducible. The intersection $E\cap H$ is the hypersurface defined by $f_d$ in $\bP^n_w$. By \cite[\S 4]{St}, the intersection $E\cap H$ has, like $E\simeq \bP^n_w$, only abelian quotient singularities. That proof  also implies our claim about $\mu$ being an embedded $\bQ$-resolution, as we show next.

By definition  $Y\subset \bC^{n+1}\times \bP_w^n$ and $\mu$ is the restriction of the projection onto the first factor. The  chart $Y_0=Y\cap \{u_0\neq 0\}$ of $Y$ is identified via the surjection
$$\phi_0:\bC^{n+1} \to Y_0$$
$$
(x_0,u_1,\ldots,u_n)\mapsto ((x_0^{w_0},x_0^{w_1}u_1,\ldots, x_0^{w_n}u_n), [1:u_1:\ldots:u_n]_w)
$$
with the quotient $\frac{1}{w_0}(-1,w_1,\ldots, w_n)$ of $\bC^{n+1}$ by the group of $w_0$-roots of unity. A similar description holds for the other charts.

In the chart $Y_0$, the exceptional divisor $E$ is given by $\{x_0=0\}$. The pullback of $f$ is given by
$$
f(x_0^{w_0},x_0^{w_1}u_1,\ldots, x_0^{w_n}u_n)=x_0^d(f_d(1,u_1,\ldots, u_n)+ x_0h(x_0,u_1,\ldots, u_n))
$$
for some polynomial $h$.
Thus the $\bQ$-normal crossings condition is satisfied  in this chart if $g(u_1,\ldots, u_n):=f_d(1,u_1,\ldots, u_n)$ is smooth on $\bC^{n}$. We check this with the Jacobian criterion. Since $(\partial g/\partial u_i)(u_1,\ldots u_n)= (\partial f_d/\partial x_i)(1,u_1,\ldots, u_n)$ for $1\le i \le n$,  smoothness of $g$ follows from the equation
$$
d\cdot f_d = \sum_{i=0}^n w_ix_i\frac{\partial f_d}{\partial x_i}
$$
together with the fact that the origin is the only singular point of $f_d$. Note that $E\cap H$ is given in $Y_0$ by the image under $\phi_0$ of the zero locus of the ideal $(x_0,g)$. Since a similar description holds in the other charts, $E\cap H$ has abelian quotient singularities since $\phi_0$ is a quotient map.

Next we note that
$$(\mu_{|Y_0}\circ\phi_0)^*(\dd x_0\wedge\ldots\wedge \dd x_n)=w_0x_0^{|w|-1}\dd x_0\wedge \dd u_1\wedge\ldots \wedge \dd u_n,$$
where $|w|=w_0+\ldots +w_n$.
 A similar description holds in the other charts.

We have now all the information needed to apply the formula of \cite{W} computing the motivic zeta function of $f$.  Since $\mu$ is an embedded $\bQ$-resolution of $f$, $Y$ is a disjoint union of strata $S_k$ characterised by
the existence of data $(G_k, \mathbf{N}_k,\boldsymbol{\nu}_k)$ satisfying the following conditions. Locally around a generic point  of $S_k$, $Y$ is analytically isomorphic to $\bC^{n+1}/G_k$ for some finite abelian  group $G_k$,  acting diagonally on the coordinates $t_0,\ldots, t_n$ of $\bC^{n+1}$ and small (i.e. not containing rotations around the hyperplanes other than the identity);
  $f\circ \mu$ is given by $t_0^{N_{0,k}}\ldots t_n^{N_{n,k}}$; and, the relative canonical divisor of $\mu$ is given by $t_0^{\nu_{0,k}-1}\ldots t_n^{\nu_{n,k}-1}$. Then by \cite[Theorem 4]{W}
\be\label{eqz}
Z^{mot}_f(s) = \bL^{-(n+1)}\sum_k [S_k]\cdot T_k(s)\cdot \prod_{i=0}^n\frac{(\bL-1)\bL^{-(N_{i,k}s+\nu_{i,k})}}{1-\bL^{-(N_{i,k}s+\nu_{i,k})}}
\ee
where $T_k(s)$ has no poles. We have showed that candidate poles from the product term in this formula contributed by $S_k$ are the zeros of
$$
1, ds+|w|, s+1, (s+1)(ds+|w|)
$$
if $S_k$ is contained in
$$
Y\setminus (E\cup H), E\setminus (E\cap H), H\setminus (E\cap H), E\cap H,
$$
respectively. This finishes the proof. $\hfill\Box$

\begin{rmk}\label{rmkIrr}  If $n=1$ and $f_d$ is not irreducible, then one has a classification up to a change of holomorphic coordinates of all possible cases for $f_d$ in \cite[Lemmas 3.3 and 3.4]{Ku}. 
\end{rmk}

\subsection{Proof of Theorem \ref{prop1}}
In the  proof of Theorem \ref{thm1} one has
$$(\mu_{|Y_0}\circ\phi_0)^*(x_0^{\beta_0}\ldots x_n^{\beta_n} \dd x_0\wedge\ldots\wedge \dd x_n)=w_0x_0^{-1+\sum_{i=0}^nw_i(\beta_i +1)}u_1^{\beta_1}\ldots u_n^{\beta_n}\dd x_0\wedge \dd u_1\wedge\ldots \wedge \dd u_n.$$
By \cite[Theorem 4]{W}, the zeta function $Z^{mot}_{f,x^\beta}(s)$ is as in (\ref{eqz}) but with the relative canonical divisor replaced by the above form. Running the proof of Theorem \ref{thm1}, we note that everything works similarly. The assumption on $x^\beta$ implies that $x_0f_d(1,u)u_1^{\beta_1}\ldots u_n^{\beta_n}$  is $\bQ$-snc, since $f_d(1,u)$ is smooth in the variables $u$ and hence its tangent cone at $\{x_0=u_1=\ldots=u_n=0\}$ must be a linear combination of the $u_i$ with $1\le i\le n$.
 $\hfill\Box$

\begin{rmk}\label{rmkK}
All the results from the introduction admit a slight generalization by replacing  in the definition  (\ref{eq2}) of the motivic zeta function the category of $\bC$-varieties with that of $k$-varieties. This holds since all the morphisms, including the group actions, in the above proofs concerning motivic zeta functions are defined over $k$.
\end{rmk}

\section{Bernstein-Sato polynomials}\label{secBS}

\subsection{$b$-functions}\label{subBfc}
Let $f,g:(X,0)\to (\bC,0)$ be germs of holomorphic functions on a complex manifold with $f(0)=0$. We set $\cO=\cO_{X,0}$ and $\cD=\cD_{X,0}$, the ring of germs of analytic functions, respectively analytic linear differential operators. We denote by $b_{f,g}(s)$ the nonzero monic generator $b_{f,g}(s)$ of the ideal of polynomials $b(s)\in\bC[s]$ of minimal degree satisfying
\be\label{eqb}
b(s)gf^s=Pgf^{s+1}\quad\quad\text{ for some }P\in \cD[s].
\ee It is a non-trivial well-known result that $b_{f,g}(s)$ is well-defined.

If $g/f$ is not holomorphic, then $b_{f,g}(s)$ is divisible by $s+1$ by Lemma \ref{lemSg}. Define in this case the {\it reduced $b$-function}
$$
\tilde b_{f,g}(s) = \frac{b_{f,g}(s)}{s+1}.
$$

When $g=1$, $b_{f,g}(s)$ (resp. $\tilde b_{f,g}(s)$) is denoted $b_f(s)$ (resp. $\tilde b_f(s)$), the usual  $b$-function (resp. reduced $b$-function) or Bernstein-Sato polynomial of $f$.

If $f,g:X\to \bC$ are regular functions on a smooth complex affine variety, one can apply the same definitions with $\cD$ replaced by the ring of global algebraic linear differential operators, and the resulting $b$-function is the lowest common multiple (well-defined due to finiteness of $b$-constant strata) of the local $b$-functions defined above for germs at points along $f^{-1}(0)$.

\subsection{Microlocal $b$-functions}\label{subMv}  Let $f:(X,0)\to (\bC,0)$ be the germ of a holomorphic function on a complex manifold with $f(0)=0$. Let $i:X\to X\times \bC$, $x\mapsto (x,f(x))$ be the graph embedding of $f$. Let $t$ be the coordinate of $\bC$. Define the rings $\cR =\cD[t,\pa_t]$, $\tilde \cR =\cR [\pa_t^{-1}]$, with $\pa_tt-t\pa_t=1$. Define the $V$-filtration on $\cR$, $\tilde \cR$ by 
$$
V^p\cR=\sum_{i-j\ge p}\cD t^i\pa_t^j
$$
and similarly for $\tilde \cR$.

Let $\cM$ be a $\cD$-module. Define
$$\cM_f=\cM\otimes_\bC\bC[\pa_t],$$
the stalk of the $\cD$-module direct image $i_+\cM$ at $(0,0)$ in $X\times\bC$. Denoting $m\otimes \pa_t^i$ by $m\pa_t^i\delta$ for a local section $m$ of $\cM$, the left $\cR$-module structure on $\cM_f$ is  defined by
$$
\xi(m\pa_t^i\delta) = (\xi m)\pa_t^i\delta - (\xi f)m\pa_t^{i+1}\delta,\quad t(m\pa_t^i\delta) = fm\pa_t^i\delta - im\pa_t^{i-1}\delta
$$
for $\xi$ a local vector field on $(X,0)$. Equivalently, $\delta$ the delta function of $t-f$. Define the algebraic microlocalization
$$\tilde \cM_f=\cM\otimes_\bC\bC[\pa_t,\pa_t^{-1}].$$ Then $\tilde \cM_f$ is a left $\tilde \cR$-module with the action defined as above. 

Assume from now that $\cM$ is holonomic. Let $u$ be a local section of $\cM_f$ (resp. $\tilde \cM_f$). The {\it $b$-function} $b_u(s)$ (resp. {\it microlocal $b$-function} $\tilde b_u(s)$) is the minimal polynomial of the action of $s:=-\pa_tt$ on $V^0\cR\cdot u/ V^1\cR\cdot u$ (resp. $V^0\tilde \cR \cdot u/ V^1\tilde \cR \cdot u$). This is a well-defined polynomial  by   \cite{K-II}, \cite{KK},  \cite{Lau} (with $\tilde b_u(s)$ defined in terms of the usual microlocalization $M\{\!\{\pa_t^{-1}\}\!\}[\pa_t]$, but this definition can be shown to be equivalent to the one above, see \cite[1.4]{S-mb}). If $\cM$ has quasi-unipotent local monodromy on  subsets forming a suitable Whitney stratification, this polynomial has rational roots.

It is known that $\delta$ can be identified with $f^s$ and $b_{m\delta}(s)$ is the monic generator of the ideal of polynomials $b(s)\in\bC[s]$ satisfying
\be\label{eqDD}
b(s)mf^s = Q(s)mf^{s+1}
\ee
for  some $Q\in\cD[s]$, for $m\in \cM$. Thus if  $\cM=\cO$ and $g\in\cO$, then 
$$
b_{g\delta}(s)= b_{f,g}(s).
$$
If $g=1$, then by \cite[Prop. 0.3]{S-mb}  the microlocal $b$-function is the reduced $b$-function:
$$
\tilde b_f(s) = \tilde{b}_{f\delta}(s).
$$ The proof can be easily adapted to yield a more general result:
 
\begin{lem}\label{lemSg} Let $\cM$ be a holonomic $\cD$-module and $m\in\cM$ a local section.

(i) If $f^{-1}m\not\in\cD m$ then $b_{m\delta}(s)$ is divisible by $s+1$.

(ii) If in addition $f$ is injective on $\cD m$, then 
$$
b_{m\delta}(s)  = (s+1)\tilde b_{m\delta}(s).
$$ 

(iii) In particular, if $g\in\cO=\cM$ and $g/f$ is not holomorphic, then $b_{f,g}(s)$ is divisible by $s+1$ and the reduced $b$-function $\tilde b_{f,g}(s)$ is the microlocal $b$-function $\tilde b_{g\delta}(s)$.

\end{lem}  
\begin{proof} (i) Setting $s=-1$ we have $b_{m\delta}(-1)mf^{-1}=Q(-1)m$  for  $Q(s)$ as in (\ref{eqDD}). Since $mf^{-1}\not\in\cD m$, $b_{m\delta}(-1)$ must be zero. 

(ii) The proof of  \cite[Lemma 1.6]{S-mb} gives  without using any of the two assumptions on $m$ that there exists $P\in\pa_t^{-1}V^0 \cR$ such that $\tilde b_{m\delta}(s)m\delta =Pm\delta$. By multiplying by $s+1=-t\pa_t$ one obtains that $(s+1)\tilde b_{m\delta}(s)$ is divisible by $b_{m\delta}(s)$.

Conversely, we show $\tilde b_{m\delta}(s)$  divides by $b_{m\delta}(s)/(s+1)$. One has $Q(s)f=tQ'(s)$ for some $Q'\in\cD[s]$ if $Q$ is as in (\ref{eqDD}). Then
$$
(s+1)\left(\frac{b_{m\delta}(s)}{s+1} - \pa_t^{-1}Q' \right)m\delta =0
$$
since $s+1=-t\pa_t$, using  (i). It is enough to show that $t$ is injective on the algebraic microlocalization $ (\cD m)[\pa_t,\pa_t^{-1}]$, since this implies that $$\left(\frac{b_{m\delta}(s)}{s+1} - \pa_t^{-1}Q' \right)m\delta=0$$
by the invertibility of $\pa_t$. Since $(\cD m)[\pa_t,\pa_t^{-1}]$ is exhausted by the filtration $F_p=\oplus_{i\le p}\cD m\pa_t^{-i}\delta$, it is enough to show that $t$ is injective on $Gr_p^F$ for all $p\in\bZ$. This is equivalent to $f$ being injective on $\cD m$.
\end{proof}

Since $\cM$ is holonomic, there exists $m\in\cM$ with $\cD m=\cM$ locally. The  filtration on $\cM_f$ (resp. $\tilde\cM _f$) defined by 
$$
 V^p\cR\cdot m\delta \quad (\text{ resp. }V^p\tilde \cR\cdot m\delta )
$$
gives rise to the decreasing {\it   $V$-filtration} $V^\al\cM_f$ (resp. {\it microlocal $V$-filtration} $V^\al\tilde \cM_f$) indexed discretely by $\al\in\bC$ with a fixed total order on $\bC$ (if $\cM$ has quasi-unipotent local monodromy one can take $\al\in\bQ$), using the decomposition of the action by $s$ on quotients $V^p/V^q$ with $p<q$, see \cite[\S 1]{S-ra}, \cite[\S 2]{S-mb}. The existence of the $V$-filtration is equivalent to the existence of $b$-functions (resp. microlocal $b$-functions). The   $V$-filtration (resp. microlocal $V$-filtration) is uniquely characterized by: 

(i) $V^p\cR\cdot V^\al \cM_f \subset V^{\al+p}\cM_f$,

(ii) $V^\al\cM_f$ are lattices of $\cM_f$, i.e. finite $V^0\cR$-modules generating $\cM_f$ over $\cR$, and

(iii) $s+\al$ is nilpotent on $Gr_V^\al\cM_f$

\noindent
(resp. similar conditions with $\tilde \cR$, $\tilde \cM_f$ replacing $\cR, \cM_f$).

One has 
$$
V^\al\cM_f=\{ u \in \cM_f\mid b_u(s)\text{ has all roots }\le -\al\}
$$
by \cite{Sab}, \cite[Cor. 1.7]{S-ra}. The same proof, relying on the unique characterization from above, and using that  $V^1\tilde \cR =\pa_t^{-1}V^0\tilde \cR$ instead of $V^1\cR=tV^0\cR$ in \cite[(1.7.1)]{S-ra}, gives: 

\begin{prop}\label{propSab} 
$
V^\al\tilde \cM_f =\{u\in \tilde \cM_f\mid \tilde b_u(s)\text{ has all roots }\le -\al\}.
$
\end{prop}

One defines the {\it microlocal $V$-filtration $\tilde V^\al\cO$} on $\cO$ by $$(\tilde V^\al\cO)\delta = (\cO\delta)\cap V^\al\tilde\cM_f$$ for $\cM=\cO$. Lemma \ref{lemSg} and Proposition \ref{propSab} imply:

\begin{cor}\label{corFalse}
$(\tilde V^\al\cO) \setminus f\cO= \{g\in\cO\setminus f\cO \mid  \tilde b_{f,g}(s)\text{ has all roots }\le -\al \}$.
\end{cor}

\subsection{Brieskorn lattices}\label{subHf} Let $f:(X,0)\to (\bC,0)$ be the germ of a holomorphic function on a complex manifold of dimension $n+1$ with  $f(0)=0$ and $f$ having an isolated singularity at $0$. There is a diagram 
$$
\xymatrix{
  & H_f'':= \frac{\Omega^{n+1}}{ \dd  f \wedge  \dd  \Omega^{n-1}} \ar@{^{(}->}[r] \ar@{->>}[d] & G_f:=H''_f[\pa_t]  \\
 \cO/(\pa f) \ar[r]^{\sim\quad}  & \Omega^{n+1}_f := \frac{\Omega^{n+1}}{\dd f\wedge \Omega^{n}}                          &}
$$ 
where the lower map is an isomorphism of $\bC$-vector spaces, the vertical map is surjective, and the top map is an inclusion. Here
$\cO=\cO_{X,0}$, $(\pa f)$ is the ideal generated by the first order partial derivatives of $f$, $\Omega^p$ consists of the germs of the holomorphic $p$-forms at the origin, $\cO/(\pa f)$ is called the {\it Milnor algebra},  $H_f''$ is called the {\it Brieskorn lattice}, and $G_f$ is called  the {\it Gauss-Manin system}.
The Brieskorn lattice is a free module of rank equal to the Milnor number $\mu_f=\dim_\bC \cO/(\pa f)$  over \( \mathbb{C}\{t\} \) and also over $\bC\{\!\{\pa_t^{-1}\}\!\}$, where the action of \( t \) is given by multiplication by \( f \) and the action of $\pa_t^{-1}$ is defined by $\pa_t^{-1}[\omega]=[\dd f\wedge\eta]$ for $\dd\eta=\omega$. The Gauss-Manin system is the localization of $H_f''$ by the action of $\pa_t^{-1}$. It is a free $\bC\{\!\{\pa_t^{-1}\}\!\}[\pa_t]$-module of rank $\mu_f$ with an action of $t$, and it is a regular holonomic $\cD$-module. Consequently it admits the rational $V$-filtration such that $\pa_tt-\al$ is nilpotent on $Gr_V^\al G_f$.    The induces a $V$-filtration on $H_f''$ and on the quotient $\Omega^{n+1}_f$. See \cite{Bri, Seb, Sa3}. On the other hand, the microlocal $V$-filtration $\tilde V^\al\cO$ defined above induces a filtration on the quotient. These two filtrations are the same:

\begin{prop}\label{propVV}\label{propVvV}\cite[Proposition 1.4]{S+}
The microlocal $V$-filtration on the Milnor algebra $\cO/(\pa f)$ agrees with the $V$-filtration on $\Omega^{n+1}_f$.
\end{prop}

\subsection{Semi-weighted homogeneous polynomials.} Assume now that $f=f_d+f_{>d}$ is a semi-weighted homogeneous polynomial of initial degree $d$ for the weight vector $(w_0,\ldots,w_n)\in\bN^{n+1}$. For $\beta\in\bN^{n+1}$ define $l(\beta)=\sum_{i=0}^nw_i(\beta_i+1)/d$. Define a new filtration $V_\omega$ on $\cO$ by setting for $\al\in\bQ$
$$
V_{w}^\al \cO = \sum_{l(\beta)\ge \al }\cO x^\beta
$$
\begin{prop}\label{propVsm}\cite[1.6]{S+}
If $f=f_d+f_{>d}$ is a semi-weighted homogeneous polynomial of initial degree $d$ for the weight vector $(w_0,\ldots,w_n)\in\bN^{n+1}$, then the $V$-filtration on $H_f''$ is the quotient of the filtration $V_w$.
\end{prop}

\begin{cor}\label{corVsm} With assumptions as in Proposition \ref{propVsm}, let $g\in \sum_{l(\beta)\ge \al }\cO x^\beta\subset \cO$ for some $\al\in\bQ$ such that 
$$g+(\pa f)\not\subset \sum_{l(\beta)> \al}\cO x^\beta+(\pa f).$$
Then $-\al$ is the biggest root of $\tilde b_{g\delta}(s)$.
\end{cor}
\begin{proof}
The assumption is equivalent in general with $[g\dd x]\neq 0$ in $Gr_V^{\al}\Omega_f^{n+1}$, by Proposition \ref{propVsm}. It follows from Proposition \ref{propVvV} that $g\delta$ is non-zero in $Gr_V^\al \tilde \cM_f$ for $\cM=\cO$ in the notation of \ref{subMv}. This implies that the maximal root of $\tilde b_{g\delta}(s)$ is $-\al$ by Proposition \ref{propSab}. \end{proof}

\subsection{Proof of Theorem \ref{thmA}}\label{subPTA} Let $h\in\bC[x_0,\ldots,x_n]$ define a semi-quasihomogeneous hypersurface singularity locally analytically isomorphic to $f$ as in Theorem \ref{thm1}. The proof of Theorem \ref{thm1} only uses local analytic properties of $f$, hence it also holds for $h$. Now, the local $b$-function of $h$ is an analytic invariant and thus equals that of $f$. By Corollary \ref{corVsm} for the monomial $g=1$, one has that  \( (s+1)(s + |w|/d) \) divides the local $b$-function of $f$. This finishes the proof for the case $n>1$. If $n=1$ one can analyze  directly the classification of $f_d$ following Remark \ref{rmkIrr} (this is also a particular case of result \cite{Lo} for general plane curves.) $\hfill\Box$

\subsection{Proof of Theorem \ref{prop2}}\label{subPT2} It follows from Theorem \ref{prop1} together with Corollary \ref{corVsm} and Lemma \ref{lemSg} (iii).
 The equivalency of the assumption in the case $f=f_d$ follows by considering the decomposition into weighted homogeneous terms, since $(\pa f)$ is generated by weighted homogeneous polynomials in this case.
$\hfill\Box$

\newpage

\begin{center}
{\bf Erratum (September 7, 2023)}
\end{center}

\bigskip

A mistake in Section \ref{secBS} 
 was pointed out by M. Saito.  We thank him for pointing this out.
 We address it here. All results in the Introduction remain true.

The definition of the microlocal $V$-filtration on $\cO$ before Corollary \ref{corFalse} should be 
$$
(\cO,\tilde V) = \Gr_F^0(\tilde M_f, V).
$$
Corollary \ref{corFalse} is false for the microlocal $V$-filtration on $\cO$,  and counterexamples are provided by M. Saito in \cite[2.2]{Sa23}. Corollary \ref{corFalse} is however not used in the rest of the paper. The mistake propagates though to Corollary \ref{corVsm} since Propositions \ref{propVvV} and \ref{propVsm} should be used with the correct microlocal $V$-filtration. Proposition \ref{propVsm} was first asserted in \cite[1.6]{S+},  see also \cite[2.2d]{Sa23}. We adjust the conclusion of Corollary \ref{corVsm}:

\begin{cor}\label{corVsm2} With assumptions as in Proposition \ref{propVsm}, let $g\in \sum_{l(\beta)\ge \al }\cO x^\beta\subset \cO$ for some $\al\in\bQ$ such that 
$$g+(\pa f)\not\subset \sum_{l(\beta)> \al}\cO x^\beta+(\pa f).$$
Then $-\al$ is a  root of $ b_{f,g}(s)$.
\end{cor}
\begin{proof}
The assumption is equivalent in general with $[g\dd x]\neq 0$ in $Gr_V^{\al}\Omega_f^{n+1}$, by Proposition \ref{propVsm}. Since $V$ on $\Omega_f^{n+1}$ is the quotient of the $V$-filtration on $H''_f$, it follows that $[g\dd x]\neq 0$ in $Gr_V^{\al}H''_f$. Then $-\al$ is a root of $b_{f,g}(s)$, cf. \cite[(2.2.4)]{Sa23}. We repeat that argument here. Let $b(s)=b_{f,g}(s)$. By assumption there is $P(s)\in\cD[s]$ such that $b(s)gf^s=P(s)fgf^s$. Equivalently,  $b(s)g\delta=P(s)fg\delta$ in $\cM_f$. The filtered de Rham complex $(\DR_X(\cM_{f}),V)$ is strict and the induced filtration is the $V$-filtration on the Gauss-Manin system $H^n(\DR_X(\cM_{f}))=G_f=H''_f[\pa_t]$. Then one has the relation $b(s)[g\dd x]=[P_0(s)fg\dd x]$ in $G_f$, where $\dd x=\dd x_0\wedge\ldots\wedge\dd x_n$ and $P_0(s)\in\cO[s]$ coincides with $P(s)$ modulo $\sum_i\pa_{x_i}\cD[s]$ by definition of $\DR_X$. Note that the class of $[g\dd x]$ is non-zero in $\Gr_V^\al G_f$ since it is non-zero in $\Gr_V^\al H''_f$ and $V$ on $H''_f$ is induced from $V$ on $G_f$. We show that  $[P_0(s)fg\dd x]\in V^{>\al}G_f$. By the nilpotency property of $V$-filtrations, some power of $s+\al$ must then divide $b(s)$, which finishes the proof. To show that  $[P_0(s)fg\dd x]\in V^{>\al}G_f$, by summing  over powers of $s$, it is enough to show
that $[hfgdx]$ lies in $V^{>\al}H''_f$ for every $h\in\cO$. This now follows
from Proposition \ref{propVsm}
on the coincidence of the $V$-filtration and $V_w$-filtration on $H''_f$ in
the semi-weighted homogeneous case, and the assumption on $[gdx]$.
 \end{proof}

We now adjust the proof of Theorem \ref{prop2} from \ref{subPT2}. 

\medskip
\noindent
{\bf Proof of Theorem \ref{prop2}.} Using Corollary \ref{corVsm2} instead of Corollary \ref{corVsm}, the theorem is proved for the case when $l(\beta)\neq 1$. 

Suppose now $l(\beta)=1$. By Theorem \ref{prop1}, it is enough to show that $-1$ is a root of multiplicity at least 2 of $b_{f,g}(s)$. Let write $f_d=\sum_\gamma c_\gamma x^\gamma$ where $x^\gamma$ are monomials of weighted degree $d$, that is, $\sum_i w_i\gamma_i/d=1$ for each $\gamma\in\bN^{n+1}$ and $c_\gamma$ are finitely many non-zero complex numbers. Then $1=l(\beta)=\sum_iw_i(\beta_i+1)/d$ implies that $\gamma_i=\beta_i+1$ for each $\gamma$. Hence up to a non-zero constant factor, $f_d=x^\gamma$ is monomial. Since $f_d$ has isolated singularities, this can only happen if $n=1$, $f_d=x_0x_1$, and $g=1$. We know that $b_{f,g}(s)=b_f(s)$ is divisible by the local $b$-function of $f$ at the origin. The latter equals the local $b$-function of $f_d$ at the origin since it is an local analytic invariant and we can choose a new set of analytic coordinates $\tilde x_0,\tilde x_1\in\bC\{x_0,x_1\}$ such that $f=\tilde x_0\tilde x_1$. The claim then follows from $b_{x_0x_1}(s)=(s+1)^2$. This finishes the proof of Theorem \ref{prop2}. $\hfill\Box$

Note that Theorem \ref{prop2} implies Theorem \ref{thmA}. The mistake pointed above does not really affect the proof of Theorem \ref{thmA} in \ref{subPTA}: one knows (without invoking the problematic Corollary \ref{corVsm} or the new Corollary \ref{corVsm2}) that $(s+1)(s+|w|/d)$ divides the local $b$-function of $f$ by the constancy of the minimal exponent in $\mu$-constant deformations and by the computation of $b$-functions of weighted homogeneous isolated hypersurfaces.

\end{document}